\documentclass[11pt]{amsart}
\usepackage{amscd}
\usepackage{graphics}
\usepackage{hyperref}

\newtheorem{thm}{Theorem}
\newtheorem{lem}[thm]{Lemma}
\newtheorem{prop}[thm]{Proposition}
\newtheorem{cor}[thm]{Corollary}
\newtheorem{defn}[thm]{Definition}

\theoremstyle{remark}

\newcommand{\an}{^{\mathrm{an}}}
\newcommand{\arxiv}[1]{\href{http://arxiv.org/abs/#1}{arXiv:#1}}
\newcommand{\Gm}{\mathbb G_m}

\newcommand{\init}{\operatorname{in}}
\newcommand{\isom}{\cong}
\newcommand{\m}{\mathfrak m}
\newcommand{\PP}{\mathbb P}
\newcommand{\p}{\mathfrak p}
\newcommand{\q}{\mathfrak q}
\newcommand{\RR}{\mathbb R}
\newcommand{\Spec}{\operatorname{Spec}}
\newcommand{\Trop}{\operatorname{Trop}}
\newcommand{\val}{\operatorname{val}}
\newcommand{\ZZ}{\mathbb Z}

\newif\ifdetails\detailsfalse

\title{The Gr\"obner stratification of a tropical variety}
\author{Dustin Cartwright}
\address{Department of Mathematics\\Yale University\\PO Box 208283\\
New Haven, CT 06520}
\begin{document}

\begin{abstract}
Each Gr\"obner stratum of a tropical variety is a connected set of points, all
of which induce the same initial subscheme.  The Gr\"obner stratification is a
coarsening of the decomposition into Gr\"obner polyhedra, and has the advantage
that it does not depend on a choice of compactification. We give an example of a
curve over a field with non-trivial valuation whose Gr\"obner stratification is
strictly finer than the coarsest polyhedral decomposition of the tropical
variety.  We also show that the Gr\"obner stratification of a
locally matroidal tropical variety is completely determined by the underlying
tropical variety.
\end{abstract}

\maketitle

Let $K$ be a field with valuation $\val \colon K^\times \rightarrow \RR$ and
let $\Gamma \subset \RR$ be the image of its valuation.
We assume and fix a splitting $r\colon \Gamma \rightarrow K^\times$ of the
valuation, which is guaranteed to exist if $K$ is algebraically
closed or if $\Gamma$ is discrete.  For any closed subscheme $V \subset \Gm^N$
and any point $w$ of $\RR^N$, it is possible to take its weighted initial scheme
$\init_w(V)$ (see Definition~\ref{def:init}). The tropicalization $\Trop(V)$ is
the set of all weights $w \in \RR^N$ such that $\init_w(V)$ is non-empty. There
exists a polyhedral complex whose support is $\Trop(V)$, but the polyhedral
structure is not unique, and there may not be a coarsest
polyhedral structure~\cite[Example~5.2]{st}.  In particular, the construction of
the polyhedral complex depends on the choice of a compactification $\Gm^N
\subset
\PP^N$ in which we take $\overline V$ to be the closure of~$V$. Then each
polyhedron of the Gr\"obner complex is the closure of the set of $w \in \RR^N$
such that $\init_w(\overline V)$ is some fixed scheme.

Without choosing a compactification, we can stratify the tropical variety based
on the initial ideal $\init_w(V)$, taken within the torus $\Gm^N$.  More
precisely, a stratum of the \emph{Gr\"obner stratification} is a component of
the set of points $w \in \RR^N$ such that $\init_w(V)$ is some fixed subscheme.
Unlike the Gr\"obner complex, each stratum need not be the relative interior of
a convex polyhedron, nor even contractible (for example, take two
tropical planes meeting in a single point), but it is at least an open subset of
an affine linear subspace of~$\RR^N$.

The purpose of this paper is to contrast the Gr\"obner stratification with
topological stratification, which depends solely on
$\Trop(V)$ as a subset of~$\RR^N$, and is defined to be
the finest stratification which is compatible with local products in the
tropical variety. More precisely, suppose that we have a linear isomorphism
$\RR^N \isom \RR^{N-m} \times \RR^m$ and an open set $U \subset \RR^N$ such that
$U \cap
\Trop(V)$ factors as a product of a subset of $\RR^{N-m}$ with an
open subset $W \subset \RR^m$.  The \emph{topological stratification} is defined to be
the finest
stratification such that for any local factoring of an open set $U$, each
stratum also factors as the product of a subset of $\RR^{N-m}$ with~$W$.  Unlike
the Gr\"obner stratification, the topological stratification depends only on
$\Trop(V) \subset \RR^N$ and not on~$V$ itself.  Without using this terminology,
Sturmfels and Tevelev gave a recipe for constructing a tropical variety whose
Gr\"obner stratification is strictly finer than its topological
stratification~\cite[Example~3.10]{st}. We give an
explicit such example as a curve over a field with non-trivial valuation in
Section~\ref{sec:example}, and
we examine the preimages of its Gr\"obner strata in the Berkovich
skeleton.

The Berkovich analytification $V^{\rm an}$ maps surjectively onto $\Trop(V)
\subset \RR^N$ by taking the valuations of the coordinate functions.  When $V$
is a curve, $V^{\rm an}$ contains a distinguished graph, the skeleton, coming
from a minimal semistable reduction of~$V$, and $\Trop(V)$ is the image of the
skeleton under a map which is linear on each edge.  Thus, there are finitely
many points in $\Trop(V)$ which are the images of the vertices of the skeleton.
Any point~$w$ whose neighborhood in $\Trop(V)$ is not a union of open
segments must be a zero-dimensional Gr\"obner stratum and also the image of a
vertex
of the skeleton.  Nonetheless, in general, the set of images of vertices of the
skeleton
neither contains nor is contained in the set of zero-dimensional Gr\"obner
strata.
Example 2.6 of~\cite{bpr} is a curve whose skeleton has a vertex which maps to a
point in the interior of a one-dimensional Gr\"obner stratum. On the other hand,
our example in Section~\ref{sec:example} has a zero-dimensional Gr\"obner
stratum contained in an open half-ray, whose preimage in the skeleton
is two disjoint edges.

Our detailed examination of this example occupies the entirety of
Section~\ref{sec:example}. In Section~\ref{sec:matroidal}, we compare the two
strata on general tropical varieties. Proposition~\ref{prop:top-charac} implies
that in any case
where the Gr\"obner refines the tropical stratification, the corresponding
initial ideals will have the same support, but different non-reduced structures.
We prove that the two strata agree for locally matroidal tropical varieties
and for curves whose multiplicities are all one (Theorem~\ref{thm:matroidal} and
Corollary~\ref{cor:curve-mult-one}).
Throughout the paper, $R$ will
denote the valuation ring of~$K$ with maximal ideal $\m \subset R$ and residue
field $k = R / \m$.  

\subsection*{Acknowledgments}
I have been supported by the National Science Foundation under award
DMS-1103856.  I would like to thank Sam Payne for many helpful suggestions,
including the question which prompted this paper.  Walter Gubler pointed me
towards Example 3.10 in \cite{st}. I also thank the Institut Mittag-Leffler and
the Max Planck Institute in Bonn for their hospitality during my work
on this paper.

\section{A non-trivial Gr\"obner stratification}\label{sec:example}

In this section, we give an example of a curve whose tropicalization has an
open ray, which contains 3 distinct Gr\"obner strata.  First, it will be useful to define the initial degeneration
$\init_w(V)$ in greater generality than has appeared in the literature. In
particular, the following does not require the entries of~$w$ to be in the value
group.
\begin{defn}\label{def:init}
For any $w \in \RR^N$, we let 
$R[x_1^\pm, \ldots,
x_N^\pm]^w$ denote the $R$-subalgebra of $K[x_1^\pm, \ldots, x_N^\pm]$ generated
by all terms $a x^u$ such that $\val(a) + u \cdot w \geq 0$, where $u \in \ZZ^N$
is the vector of exponents.
We define
$V^w$ to be the closure of $V$ in $\Spec R[x_1^\pm, \ldots, x_N^\pm]^w$. We then
define a ring homomorphism $\phi\colon R[x_1^\pm, \ldots, x_N^\pm] \rightarrow
k[x_1^\pm, \ldots, x_N^\pm]$ by sending the monomial $a x^u$ to
\begin{equation*}
\phi(a x^u) = \begin{cases}
\overline{a r(w \cdot u)} x^u & \mbox{if } \val(a) + r \cdot u = 0 \\
0 & \mbox{if } \val(a) + r \cdot u > 0,
\end{cases}\end{equation*}
where $\overline{a r(w \cdot u)}$ denotes the image of $a r (w \cdot u)$ in the residue field~$k$.  The
\emph{initial degeneration} $\init_w(V)$ is defined to be $(\phi^*)^{-1}(V^w)$,
where $\phi^*$ is the map of schemes induced by~$\phi$.
\end{defn}
If the entries of $w$ lie in the value group, then $\phi^*$ induces an
isomorphism from $\init_w(V)$ to the special fiber of $V^w$, and if not,
$\init_w(V)$ can be thought of as
the special fiber of $V^w$ after extending
$K$ to a field whose valuation group contains the coordinates of~$w$.

For our example, we let $\pi \in K$ be an element with non-trivial valuation,
and for convenience, we rescale the valuation such that $\val(\pi) = 1$.
Throughout this section, $I \subset K[x^\pm, y^\pm, z^\pm]$ will denote the
ideal generated by the $2 \times 2$ minors of the matrix
\begin{equation}\label{eq:def-i}
\begin{bmatrix}
x-1 & \pi(y-1) & \pi(y-1 - \pi z) \\
\pi (y-1) & \pi(y-1 - \pi z) & x-1-\pi
\end{bmatrix},
\end{equation}
and $V$ will be the subvariety of $\Gm^3$ defined by $I$.  Abstractly, $V$ is
isomorphic to $\PP^1$ with a finite number of points removed, and it can be
given explicitly as the image of the rational map
\begin{equation*}
(x,y,z) = \left(
\frac{\pi}{1-u^3} + 1,
\frac{u}{1-u^3} + 1,
\frac{u}{\pi(1+u+u^2)}
\right).
\end{equation*}
We will use $X$, $Y$, and~$Z$ as coordinates for $\RR^3$.  The intersection of
the tropical variety of~$V$ with the half space defined by $Z > -1$ is the
single half ray with $X = Y = 0$. The initial stratification divides this
half-ray into three strata, based on the ideal defining $\init_w(V)$:
\begin{equation}\label{eq:initial-ideals}
I\big(\init_{(0,0,Z)}(V)\big) = \begin{cases}
\langle (x-1), (y-1)^2 \rangle & \mbox{if } -1 < Z < 0 \\
\langle (x-1)^2, (x-1)z-(y-1) \rangle & \mbox{if } Z = 0 \\
\langle (x-1)^2, (y-1) \rangle & \mbox{if } 0 < Z.
\end{cases}
\end{equation}
Note that all three of these initial schemes have the same support, namely the
subtorus defined by $x = y = 1$, but different non-reduced structures.

\begin{table}
\begin{tabular}{l|c|c}
Component $D$ & $(\val_D(x), \val_D(y), \val_D(z))$ & $\Trop(D)$ \\ 
\hline
$u^3 - \pi - 1$ & $(1, 0, 0)$ & $(\infty, -1, *)$ \\
$u^3 - u - 1$ & $(0, 1, 0)$ & $(0, \infty, -1)$  \\
$u$ & $(0, 0, 1)$ & $(0, 0, \infty)$ \\
$u^2 + u + 1$ & $(1, 1, 1)$ & $(-\infty, -\infty, -\infty)$ \\
$u - 1$ & $(1, 1, 0)$ & $(-\infty, -\infty, -1)$ \\
$\infty$ & $(0, 0, 1)$ & $(0, 0, \infty)$
\end{tabular}
\caption{Boundary components of the compactification of~$V$, over an
algebraically closed field of characteristic not $3$ or~$23$. The value of~$*$
is $-1$ for the root of $u^3-\pi - 1$ such that $\val(u-1) = 1$ and $-2$ for the
other two roots, for which $\val(u-1) = 0$.}
\label{tbl:boundary}
\end{table}
We now compute the skeleton of $V\an$ and show that the map from this skeleton
to $\Trop(V)$ does not distinguish $(0,0,0)$ from nearby points.  We
will assume that $K$ is algebraically closed and that $k$ does not have
characteristic~$3$
or~$23$.  The purpose of these two assumptions is so that the boundary $\PP^1
\setminus V$ consists of exactly $11$ distinct closed points which are the roots
of the
polynomials shown in the first column of Table~\ref{tbl:boundary}.  In
characteristics $3$ and~$23$, there are extra coincidences among the divisors in
Table~\ref{tbl:boundary}, which change the skeleton, but not the part mapped to
the half-space $Z > -1$.  There is a unique $R$-model of $V$ defined by taking
$u$ to be
a non-constant rational function on the special fiber.
This model is not semistable because each root of $u^3 - \pi - 1$
intersects a corresponding root of $u^3 - 1$ in the special fiber. If we blow up
the special fiber at each of the three roots of $u^3 - 1$, we obtain a
semistable model whose special fiber has $4$ components. The original component
maps to $(0, 0, -1)$. The exceptional divisor of the blow-up at $u=1$
maps to $(0, -1, -1)$ and the other two exceptional divisors map
to $(0, -1, -2)$.  The dual graph of the special fiber is a claw graph and the
Berkovich skeleton consists of this graph together with $11$ infinite edges. The
map of the Berkovich skeleton to $\RR^3$ is shown in Figure~\ref{fig:example}.

\begin{figure}
\includegraphics{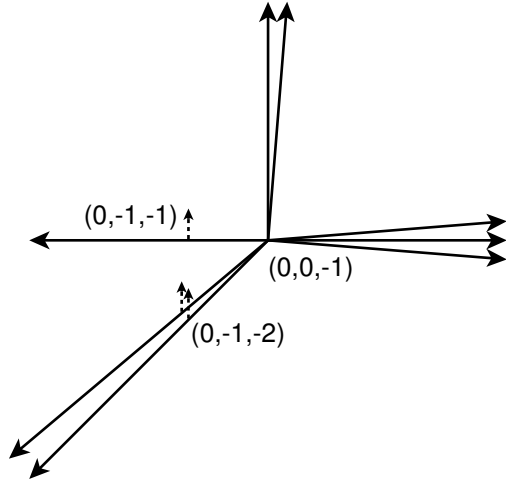}
\caption{Skeleton of the Berkovich analytification of $V$, shown using its
projection onto the $Y$ and $Z$ coordinates. 
Pairs and triples of lines which map to the same place
are drawn slightly separated from each other.
The short dotted lines represent line segments pointed into the positive $X$
direction. The other lines lie in the $X=0$ plane except for the two
southwest-most rays, which point in the $(-1, -1, -1)$ direction.}
\label{fig:example}
\end{figure}

The preimage of the half-plane $Z > -1$ in the skeleton consists of two disjoint
rays. Since $V\an$ has a deformation retract onto its skeleton, this implies
that $V\an$ itself decomposes into two connected components in the corresponding
analytic domain $D \subset \Gm^3$ defined by $\val(z) > -1$. We now compute this
decomposition explicitly.  Since $I$ is generated by all $2 \times 2$ minors of
(\ref{eq:def-i}), it also contains the determinant
\begin{multline*}
\frac{1}{\pi^2}\begin{vmatrix}
1 + \pi z - y & 0 & y - 1 \\
x-1 & \pi(y-1) & \pi(y-1 - \pi z) \\
\pi (y-1) & \pi(y-1 - \pi z) & x-1-\pi
\end{vmatrix}  \\
\ifdetails
= \frac{1}{\pi} (1 + \pi z - y) (y - 1) (x-1 - \pi)
- (1 + \pi z - y) (y-1 - \pi z) ^2 \\
+ \frac{1}{\pi} (y-1) (x-1) (y - 1 - \pi z)
- (y-1)^3 \\
\shoveright{= (y-1)(y-1-\pi z) - (y-1)^3 + (y-1 - \pi z)^3}  \\
\fi
= (1-3 \pi z)(y-1)^2 + \pi z(3 \pi z - 1) (y-1) - \pi^3 z^3
\end{multline*}
This is the equation of the projection of $V$ onto the $y$ and $z$ coordinates.
We now assume that $K$ does not have characteristic~$2$ in order to use the
quadratic formula to find analytic representations of the two
branches of $y$ in terms of $z$:
\begin{align}
\ifdetails
y
&= 1 + \frac{-\pi z(3 \pi z - 1) \pm \sqrt{\pi^2z^2(3 \pi z - 1)^2 + 4 \pi^3
z^3(1 - 3\pi z)}}{2(1-3\pi z)} \notag \\
&= 1 + \frac{\pi z}{2} \left( 1 \pm \sqrt{1 + \frac{4 \pi z}{1-3\pi z}}\right)
\notag \\
\fi
\ifdetails\else y \fi
&= 1 + \frac{\pi z}{2}\left(1 \pm \sqrt{\frac{1 - 3 \pi z}{1 + \pi z}}\right),
\label{eq:y}
\end{align}
which can be expanded as a power series in $z$, convergent on the analytic disk
defined by $\val(z) > -1$.
Using the left-most minor of (\ref{eq:def-i}), we
have
\begin{equation*}
x - 1 = \frac{\pi (y-1)^2}{y-1-\pi z},
\end{equation*}
and we can use the expansion of (\ref{eq:y}) as a power series to get analytic
formulas:
\begin{equation}\label{eq:x}\begin{split}
x &= 1 -\pi + 2 \pi^2z + \cdots \\
x &= 1 -\pi^4z^3 + \pi^6 z^5 + \cdots
\end{split}\end{equation}
for the positive and negative branch of $y$ respectively.  Thus, (\ref{eq:y})
and~(\ref{eq:x}) give equations for each branch of $V$ as the graph of an
analytic function from the $z$ coordinate to the $x$ and~$y$ coordinates.
Although we only defined $\init_w(V)$ for subschemes of the torus, we can also
take initial ideals in the analytic domain~$D$. For $X = Y=0$ and $Z > -1$, the
formulas (\ref{eq:y}) and~(\ref{eq:x}) give $x-1 = y-1=0$. Thus, either branch
on its own degenerates to the subtorus defined by $x=y=1$, and it is only when
the two branches degenerate together that the ``twisted'' non-reduced structure
appears in~(\ref{eq:initial-ideals}).

An important use of the polyhedral complex decomposition of a tropical variety
is the theory of tropical compactifications, as introduced by
Tevelev~\cite{tevelev} and later extended to fields with non-trivial valuation by
Qu~\cite{qu} (see also~\cite{lq}) and Gubler~\cite[Sec.~11]{gubler}. The topological
stratification of $\Trop(V)$ has polyhedral strata, and we can compactify $V$ by taking its closure inside
the corresponding toric scheme. This compactification cannot be a
tropical compactification since a tropical compactification must come from a
polyhedral complex which refines the Gr\"obner
stratification~\cite[Cor.~12.9]{gubler}.
We examine this compactification explicitly.

A toric scheme over $R$ corresponds to a fan in $\RR^3 \times \RR_{\geq
0}$~\cite[Sec.~7]{gubler}, and we will look at one affine chart of the scheme
coming from the cone over the topological stratification of $\Trop(V)$, placed
at height~$1$.
Our usual ray of
$\Trop(V)$ produces the cone generated by
the vectors $(0,0,1,0)$ and $(0,0,-1,1)$. The corresponding affine toric scheme
is $\Spec R[x^\pm, y^\pm, v]$, where the subtorus $\Spec K[x^\pm, y^\pm, z^\pm]$
is
defined by $v = \pi z$. Near the point in the special fiber defined by $x=y=1$
and $v=0$, the closure $\overline V \subset \Spec R[x^\pm, y^\pm, v]$ of $V$ has
two branches, given by substituting $w = \pi z$ into
the analytic equations (\ref{eq:y}) and~(\ref{eq:x}).
\begin{align*}
y &= 1 + \frac{w}{2} \left( 1 + \sqrt{\frac{1 - 3 w}{1 + w}} \right)  &
y &= 1 + \frac{w}{2} \left( 1 - \sqrt{\frac{1 - 3 w}{1 + w}} \right) \\
x &= 1 -\pi + 2\pi w + \cdots &
x &= 1 -\pi w^3 + \pi w^4 + \cdots
\end{align*}
Thus, $\overline V$ consists of two analytic branches, each isomorphic to the
completion of the two-dimensional scheme $\Spec R[w]$, meeting at the single
closed point $(x,y,w)=(1,1,0)$. This is a
textbook example of a non-Cohen-Macaulay singularity, which was the ingredient
suggested in~\cite[Ex.~5.2]{st} for building a tropical variety with a
polyhedral structure which is too coarse to be tropical.
As they argued, the non-Cohen-Macaulay point, together with Theorem~1.2
of\cite{tevelev} give another proof that this (partial) compactification is not
tropical.

\section{Locally matroidal tropical varieties}\label{sec:matroidal}

In this section, we more systematically compare the Gr\"obner and topological
stratifications.  Proposition~\ref{prop:top-charac} relates the topological
stratification to the structure of the initial ideals.  The main result of this
section is Theorem~\ref{thm:matroidal}, which gives criterion for the two
stratifications to agree.

We first need several lemmas, which extend standard results to initial ideals
at points whose coordinates are not necessarily in the value group. In
particular, Lemma~\ref{lem:degen} tells us that even when $\init_w(V)$ is not
the special fiber of~$V^w$, it can be realized as
the special fiber of a flat family over some extension of~$R$.

\begin{lem}\label{lem:base-change}
Let $\widetilde K \supset K$ be an extension of valued fields and $w$ any point in
$\RR^N$. Then $\init_w(V \times \Spec\widetilde K) = \init_w(V) \times \Spec\widetilde k$ as
subschemes of the torus, where $\widetilde k$ is the residue field of $\widetilde K$.
\end{lem}

\begin{proof}
We let $\widetilde V^w$ be the closure of the extension $V \times \Spec\widetilde K$ in the
tilted torus $R[x_1^\pm, \ldots, x_N^\pm]^w$, as in
Definition~\ref{def:init}. We have a commutative diagram:
\begin{equation*}\begin{CD}
\widetilde R[x_1^\pm, \ldots, x_N^\pm]^w
@>\widetilde\phi>> \widetilde k[x_1^\pm, \ldots, x_N^\pm] \\
@AAA @AAA \\
R[x_1^\pm, \ldots x_N^\pm]^w @>\phi>> k[x_1^\pm, \ldots, x_N^\pm],
\end{CD}\end{equation*}
where $\phi$ and~$\widetilde\phi$ are as in Definition~\ref{def:init}.
By
\cite[Thm.~A.3]{op}, the ideal of $\widetilde V^w$ is the image of the ideal
of~$V^w$ under the inclusion on the left (although \cite{op} assumes that $K$ is
algebraically closed, this isn't used in the proof of Theorem~A.3). Thus, the
ideal of $\init_w(V \times \Spec
\widetilde K)$ is the image of the ideal of $V^w$ by the upper-left path around
the
square. The ideal of $\init_w(V) \times \Spec \widetilde k$ is the image by
the lower right path, and so the two schemes are equal.
\end{proof}

\begin{lem}\label{lem:degen}
Let $w$ be a point in $\Trop(V)$. Then there exists a valuation ring $\widetilde
R$ dominating $R$ and a flat family over $\widetilde R$ whose special fiber is
isomorphic to $\init_w(V)$ and whose general fiber is $V \times \Spec \widetilde
K$, where $\widetilde K$ is the fraction field of $\widetilde R$.
\end{lem}

\begin{proof}
We build the extension $\widetilde K \supset K$, together with a splitting of
its valuation homomorphism one coordinate of~$w$ at a time. Let $K_0$
and~$\Gamma_0$ equal $K$ and $\Gamma$ respectively, and we inductively define
$K_i$ as follows. If no multiple of $w_i$ lies in $\Gamma_{i-1}$, then we take
$K_i$ to be the transcendental extension $K_{i-1}(t_i)$ with the unique
extension of the valuation given by $\val(t_i) = w_i$. Otherwise, if $m
\val(w_i)$ is in $\Gamma_{i-1}$, with $m$ minimal, then $K_i$ is constructed as
the extension $K(t_i)$, where $t_i$ is a primitive $m$th root of $r(m \val(w_i))$. In both
cases, we also extend the splitting, which by abuse of notation, we still call
$r$, by setting $r(w_i) = t_i$. In the end, we have produced a valued field
$\widetilde K = K_N$ whose value group $\widetilde\Gamma = \Gamma_N$ includes
the entries of $w$.
We also note that at each
stage of our construction of $\widetilde K$, the residue field of the valuation
ring is unchanged.

The desired family will be $\widetilde V^w$, which is flat and finite presentation over
$\Spec \widetilde R$~\cite[Prop.~A.1]{op}.  By Lemma~\ref{lem:base-change}, we
know that $\init_w(V \times \Spec \widetilde K)$ is $\init_w(V)$, and so it only
remains to show that the map $\widetilde\phi$ as in Definition~\ref{def:init}
induces an isomorphism between $\init_w(V)$ and the special fiber of $V^w$.
First, since $w_i \in \widetilde\Gamma$, each variable $x_i$ of
$k[x_1^\pm,\ldots, x_N^\pm]$ is the image of
$r(-w_i) x_i$ under $\widetilde\phi$, so $\widetilde\phi$ is surjective. Second,
for any monomial $a x^u$ in the kernel of $\widetilde\phi$, i.e.\ with $\val(a)
+ w \cdot u > 0$, then $w \cdot u$ is in $\widetilde\Gamma$, so there exists $b
\in \widetilde R$, with $\val(b) = -w \cdot u$. Thus, $b x^u \in R[x_1^\pm,
\ldots, x_N^\pm]$, and $a/b \in \widetilde R$ with positive valuation, so $a
x^u$ vanishes on the special fiber. We've shown that the special fiber of
$\widetilde V^w$ is $\init_w(\widetilde V) = \init_w(V)$, so $\widetilde V^w$
forms the desired family.
\end{proof}

\begin{lem}\label{lem:link}
Let $w$ be a point in $\Trop(V)$. Then for sufficiently small vectors~$\delta$,
$\init_{w + \delta}(V) = \init_{\delta}(\init_w(V))$, where the last initial
ideal is taken with respect to the trivial valuation on~$k$.
\end{lem}

\begin{proof}
By choosing a compactification $\Gm^N$ into $\PP^N$, we can partition $\RR^N$
into a Gr\"obner complex as in~\cite[Sec.~10]{gubler}. Now we claim that the
lemma is true if we take $\delta$ sufficiently small such that any polyhedron
containing $w + \delta$ also contains $w$. If so, we can choose an extension
$\widetilde K \subset K$ whose valuation group contains the coordinates of $w$
and $w+\delta$ as in
the proof of Lemma~\ref{lem:degen}. Then Proposition~10.9 of~\cite{gubler} tells
us that $\init_{w+\delta}(V_{\widetilde K}) = \init_\delta(\init_w(V_{\widetilde
K}))$, and Lemma~\ref{lem:base-change} tells us that the same equality holds
over~$K$.
\end{proof}

\begin{prop}\label{prop:charac}
A point $w \in \Trop(V)$ is in a Gr\"obner stratum of dimension~$m$ if
and only if the maximal subtorus preserving $\init_w(V)$ has dimension~$m$.
\end{prop}

\begin{proof}
For sufficiently small vectors $\delta$, the initial degeneration $\init_{w +
\delta}(V)$ equals the initial degeneration $\init_{\delta}(\init_w(V))$
by Lemma~\ref{lem:link}. Therefore, $w$ lies in a $m$-dimensional Gr\"obner
stratum if and only if $\init_\delta(\init_w(V))$ is equal to $\init_w(V)$ for
all $\delta$ in a $m$-dimensional vector space.
The scheme $\init_w(V)$ is unchanged by taking initial with respect to some
rational weight vector $\delta$ if and only if $\init_w(V)$ is invariant under
the action of the corresponding one-dimensional subtorus of $\Gm^N$. Thus, the
vector space of all weights under which $\init_w(V)$ is unchanged corresponds to
the maximal subtorus under which $\init_w(V)$ is invariant, and in particular,
they have the same dimension.
\end{proof}

The proof of Proposition~\ref{prop:charac} already shows that the Gr\"obner
stratification refines the topological one. However, it can also be seen through
the following characterization of the topological stratification in terms of
initial ideals.

\begin{prop}\label{prop:top-charac}
A point $w \in \Trop(V)$ is in the topological stratum of dimension~$m$ if and
only if the maximal subtorus preserving the reduced subscheme of $\init_w(V)$
has dimension~$m$.
\end{prop}

\begin{proof}
First, let $W$ denote the reduced induced subscheme of $\init_w(V)$ and suppose
that $W$ is preserved by some $m$-dimensional subtorus, and thus the tropical
variety of~$W$ is a product with the vector space corresponding to the
subtorus. Since taking reduced subschemes does not change the underlying set of
a tropical
variety, the tropical variety of $\init_w(V)$ now has the same property. Thus,
$\Trop(V)$ is also a product with a vector space in a neighborhood of~$w$
by Lemma~\ref{lem:link}, so $w$ is in an $m$-dimensional topological stratum

Conversely, suppose that $\Trop(V)$ is a product with a $m$-dimensional linear
space in a neighborhood of~$w$.  Then $\Trop(\init_w(V))$ is globally a product
with a vector space, and we can change coordinates to assume that the vector
space generated by the first $m$ coordinate vectors.  If we let $p$ be the
projection of $\Gm^N$ onto the last $N-m$ coordinates, then
$\Trop(p(\init_w(V)))$ is the projection of the $\Trop(\init_w(V))$ onto the
last
$N-m$ coordinates.  Therefore, the Bieri-Groves theorem tells us that
$p(\init_w(V))$ has dimension $d-m$,
where $d$ is the dimension of~$V$.
For any irreducible component~$V'$
of $\init_w(V)$, the projection $p(V')$ must have dimension at most $d-m$,
so $p^{-1}(p(V'))$ has dimension at most~$d$.  However, $V'$ is a
$d$-dimensional variety and is contained in $p^{-1}(p(V'))$, so they must be
equal.  Therefore, $V'$ is invariant under the subtorus which acts on the first
$m$ coordinates, as is the reduced subscheme of $\init_w(V)$.
\end{proof}

\begin{cor}
If $\Trop(V)$ is a linear space in a neighborhood of a point~$w$, then the
support of $\init_w(V)$ is a finite union of torus orbits.
\end{cor}
Recall that at a point $w$ in the interior of a maximal cone of the Gr\"obner
complex, the multiplicity is the sum of the multiplicities of all primary
components of $\init_w(V)$.  By the balancing condition, the multiplicities in
any neighborhood of a locally linear point will all be equal. However, we can
given an intrinsic definition of the multiplicity at any locally linear point
$w$ as the sum
\begin{equation*}
\sum_{W \in \operatorname{Min}(\init_w(V))}
\operatorname{length}(\mathcal O_{V,W})
\left[ K(W)^T : k\right]
\end{equation*}
where $\operatorname{Min}(\init_w(V))$ is the set of irreducible components
of~$\init_w(V)$, and $K(W)^T$ is the field of rational functions on $W$ which
are invariant under the maximal subtorus of $\Gm^N$ which preserves~$W$.

The tropicalizations of linear spaces have a purely combinatorial
description in terms of the associated matroid~\cite{ak}, which we recall here.
Let $M$ be a matroid with ground set $[N]$ and bases $B_1, \ldots, B_r$. For any
basis $B = \{b_1, \ldots, b_m\}$ and $w \in \RR^N$, let $w_B$ be the sum
$w_{b_1} + w_{b_2} + \cdots + w_{b_m}$. Let $M_w$ be the set of bases $B$
for which $w_B$ is minimized, and it can be shown that $M_w$ is also the set of
bases of a matroid. The Bergman fan of $M$ is the set of vectors $w$
for which $M_w$ has no loops.
We say that a tropical variety is \emph{locally matroidal} at a point $w \in
\RR^N$ if
there exists a neighborhood of $w$ in which all multiplicities are~$1$ and which
is equal to a neighborhood of the Bergman fan of some matroid $M$, possibly
after a change of coordinates from an element of
$\operatorname{GL}_N(\ZZ)$.
Locally matroidal fans have been studied before as analogues of smooth
varieties, for example in~\cite{shaw} and~\cite[Sec.~7]{ks}.

\begin{thm}\label{thm:matroidal}
Let $w$ be locally matroidal point of $\Trop(V)$.
Then, in appropriate coordinates, $\init_w(V)$ is a
linear subspace of~$\Gm^N$. In particular, the
Gr\"obner stratification and the topological stratification agree in a
neighborhood of~$w$.
\end{thm}

\begin{cor}\label{cor:curve-mult-one}
If $V$ is a curve and all multiplicities of $\Trop(V)$ are~$1$, then
the Gr\"obner and topological stratifications on $\Trop(V)$ coincide.
\end{cor}

\begin{proof}
Since the Gr\"obner stratification refines the topological stratification, it
suffices to prove that each one-dimensional topological stratum of
$\Trop(V)$ is
also a Gr\"obner stratum. However, any one-dimensional topological stratum is
locally linear by definition, and thus matroidal, so the claim follows from
Theorem~\ref{thm:matroidal}.
\end{proof}

Recall that a  subvariety~$V$ of a torus is called \emph{sch\"on} if the initial
degeneration $\init_w(V)$ is smooth for every $w \in \RR^N$

\begin{cor}\label{cor:schoen}
If $\Trop(V)$ is locally matroidal at all points $w \in \Trop(V)$, then $V$ is
sch\"on.
\end{cor}

One ingredient in the proof of Theorem~\ref{thm:matroidal} is the fact that
linear spaces are characterized by their tropicalizations, which was proved by
Katz and Payne~\cite[Prop.~4.2]{kp}.  Given this characterization, the
remaining difficulty is to prove that the $\init_w(V)$ is reduced, for which we
use the following.

\begin{thm}\label{thm:hironaka}
If $\init_w(V)$ is generically reduced and its reduced induced subscheme is
geometrically integral and normal, then $\init_w(V)$ is reduced.
\end{thm}

Theorem~\ref{thm:hironaka} is a generalization of a result first obtained by
Hironaka for degenerations over the localization of a ring of finite type over a
field~\cite[Lemma~4]{hironaka} (see also \cite[Lemma~III.9.12]{hartshorne}).
We use the Hironaka's argument for the case when $R$ is an excellent DVR,
and use Noetherian approximation to reduce to this case when $R$ is a general
rank~$1$ valuation.

\begin{proof}[Proof of Theorem~\ref{thm:hironaka}]
Using Lemma~\ref{lem:degen}, we let $S$ denote the coordinate ring of the flat
family over the valuation ring, for which we drop the tilde and write~$R$. We
let $\p$ denote the unique minimal prime of $S / \m S$.

We first prove the theorem under the assumption that $R$ is an excellent
discrete valuation ring and write $\pi$ to denote a uniformizer of~$R$. We've
assumed
that $S/\m S$ is generically reduced, so
$\pi$ generates the maximal ideal in $S_\p$, meaning that $S_\p$ is a
discrete valuation ring. Therefore, $\widetilde S$, the normalization of~$S$ is a
subring of $S_\p$, and so $\widetilde S / \p \widetilde S$ is a subring $S_\p/ \p S_\p$.
However, $S_\p / \p S_\p$ is the field of fractions of $S/\p$, which we've
assumed to be normal, so $\widetilde S / \p \widetilde S$ must equal $S / \p S$.

Since $S/ \m S$ is generically reduced, it will be sufficient to
show that it has no embedded primes. Suppose that $\q$ is any prime containing
$\p$ and we will show that $\q$ is not an embedded prime of~$S / \m S$.
Since $R$ is excellent,
then, in particular, the normalization $\widetilde S$ is a finite $S$-module.
We've shown that $\widetilde S / \q \widetilde S$ equals $S / \q S$, so
Nakayama's Lemma implies that $\widetilde S_\q$ equals $S_\q$.
In other words, $S_\q$ is normal.
Thus, the principal ideal in~$S_\q$ generated by $\pi$ cannot have any embedded
primes, and thus $\q$ cannot be an embedded prime of $S/\m S$. This completes the proof when $R$ is an excellent DVR.

Now, we allow $R$ to be an arbitrary valuation ring and we reduce to the
excellent case by Noetherian approximation. By~\cite[Cor.~11.2.7]{ega4-3}, there
exists a finite type $\ZZ$-algebra $R_0$, contained in $R$, and a finite type,
flat $R_0$-algebra $S_0$ such that $R_0 \subset R$ and $S \isom R \otimes_{R_0}
S_0$ as $R$-algebras.  Set $R_1$ to be the localization $(R_0)_{\m \cap R_0}$
and let $\m_1$ be its maximal ideal.  By \cite[Prop.~7.1.7]{ega2},
there exists a discrete valuation ring $(R_2, \m_2)$ dominating $R_1$ with the
same field of fractions. From the proof of that proposition, we see that $R_2$
is the normalization of a localization of a ring of finite type over $R_0$, so
$R_2$ is excellent~\cite[Schol.~7.8.3(iii)]{ega4-2}.  These various bases are
summarized in the following commutative diagram:
\begin{equation}\label{eq:cd}
\begin{CD}
S_0 \otimes_{R_0} R_2 @<<< S_0 \otimes_{R_0} R_1 @>>> S = S_0 \otimes_{R_0}
R\hspace{-0.8in} \\
@AAA @AAA @AAA \\
R_2 @<<< R_1 @>>> R
\end{CD}\hspace{0.8in}\end{equation}
We will write $S_0 \otimes_{R_0} R_i$ as $S_i$ for $i=1, 2$.

The rings in the bottom row of~(\ref{eq:cd}) are local rings and we now wish to
the corresponding special fibers of the rings in the top row.  The inclusions of
$R_1$ into $R$ and~$R_2$ are local homomorphisms, so if we take the quotients of
the rings in (\ref{eq:cd}) by the maximal ideals in the bottom row, we get a
commutative diagram whose squares are tensor products:
\begin{equation*}\begin{CD}
S_2/ \m_2 S_2 @<<< S_1 / \m_1 S_1 @>>> S / \m S \\
@AAA @AAA @AAA \\
R_2 / \m_2 @<<< R_1 / \m_1 @>>> R/ \m
\end{CD}\end{equation*}
By hypothesis, $S/ \m S$ has a unique minimal prime.
Thus, the only minimal prime of $S_1 / \m_1 S_1$ is
the contraction $\p \cap (S_1 / \m_1 S_1)$~\cite[Prop.~4.2.7(i)]{ega4-2}, and we
will call this prime~$\p_1$.
Moreover, localizing at this prime kills off any embedded components in $S
/ \m S$, so the localization of $S_1 / \m_1 S_1$ at $\p_1$ is a subring of an
integral domain and thus $S_1/ \p_1$ is generically reduced. Since $S / \p$ is
geometrically integral, then
$S_2 / \p_1 S_2$ is integral and $\p_1 S_2$ is the unique minimal prime over
$\m_2 S_2$. Therefore, we can apply the proof of the excellent case to conclude
that $\m_2 S_2$ is radical. Since $S_2 / \m_2 S_2$ has no embedded primes, then
neither do $S_1 / \m_1 S_1$ and $S / \m S$~\cite[Prop.~4.2.7(i)]{ega4-2},
which is what we wanted to prove.
\end{proof}

\begin{proof}[Proof of Theorem~\ref{thm:matroidal}]
The degeneration $\init_w(V)$ and its reduced subscheme
have the same tropicalizations as sets.
In general, the tropicalization of the reduced subscheme could have lower
multiplicities, but 
the fact that all the multiplicities are~$1$ means that
$\init_w(V)$ and its reduced subscheme have the same multiplicities as well.
Therefore, the reduced subscheme of $\init_w(V)$ is isomorphic to the complement
of a hyperplane arrangement in projective space~\cite[Prop.~4.2]{kp}, and in
particular it is normal. Thus,
Theorem~\ref{thm:hironaka} implies that $\init_w(V)$ is reduced.
The initial degenerations of a linear space are also linear spaces, so by
Lemma~\ref{lem:link}, the initial subschemes are also reduced sufficiently close
to~$w$.
Therefore, the criteria in Propositions~\ref{prop:charac}
and~\ref{prop:top-charac} agree, so the Gr\"obner and topological
stratifications agree in this neighborhood.
\end{proof}

\end{document}